\documentclass{article}

\usepackage{amsmath,amssymb,amsthm}
\usepackage{xurl}
\usepackage[hidelinks]{hyperref}

\title{A counterexample to the symmetric-maximizer conjecture
for Lyapunov operators}

\newtheorem{theorem}{Theorem}
\newtheorem{corollary}[theorem]{Corollary}
\newtheorem{conjecture}[theorem]{Conjecture}
\theoremstyle{remark}
\newtheorem{remark}[theorem]{Remark}

\author{
Daniel Kressner\thanks{Institute of Mathematics, EPFL, Lausanne, Switzerland.
\href{mailto:daniel.kressner@epfl.ch}{%
    \nolinkurl{daniel.kressner@epfl.ch}}}
\and
Bart Vandereycken\thanks{%
  Section of Mathematics, University of Geneva, Geneva, Switzerland.\\
  \href{mailto:bart.vandereycken@unige.ch}{%
    \nolinkurl{bart.vandereycken@unige.ch}}%
}
}

\begin{document}

\raggedbottom

\maketitle

\begin{abstract}
It has been conjectured that the operator norm of the Lyapunov operator induced by the
Frobenius norm is always attained at a symmetric matrix.  The conjecture is
known to hold for all matrices of order at most five.  We give an integer matrix of
order seven for which the skew-symmetric restricted norm is strictly larger
than the symmetric restricted norm.  A rational separator and exact-arithmetic
certificates establish the strict inequality without relying on floating-point
computations.  A direct-sum construction yields counterexamples in every order
$n\geq 7$; the case $n=6$ remains open.
\end{abstract}

\medskip

\section{Introduction and mathematical setting}

Given $A\in\mathbb{R}^{n\times n}$, define the Lyapunov operator
\[
    \mathcal{L}_A:\mathbb{R}^{n\times n}\to
    \mathbb{R}^{n\times n},
    \qquad
    \mathcal{L}_A(X)=AX+XA^\top.
\]
We equip the space of matrices with the Frobenius inner product and norm,
\[
    \langle X,Y\rangle_F=\operatorname{tr}(X^\top Y),
    \qquad
    \|X\|_F=\langle X,X\rangle_F^{1/2},
\]
and consider the induced operator norm
\begin{equation}\label{eq:norm}
    \|\mathcal{L}_A\|
    :=\max_{X\ne0}\frac{\|AX+XA^\top\|_F}{\|X\|_F}.
\end{equation}
After column-wise vectorization, $\|\mathcal{L}_A\|$ is also the spectral norm of
$I_n\otimes A+A\otimes I_n$.

Define the orthogonal subspaces
\[
    \mathbb{S}_n=\{X\in\mathbb{R}^{n\times n}:X^\top=X\},
    \qquad
    \mathbb{K}_n=\{X\in\mathbb{R}^{n\times n}:X^\top=-X\}.
\]
Both subspaces are
invariant under $\mathcal{L}_A$ since
\[
    {\mathcal{L}_A(X)}^\top=\mathcal{L}_A(X^\top).
\]
Relative to this orthogonal decomposition, $\mathcal{L}_A$ is therefore block
diagonal and its operator norm is the maximum of the
norms of these blocks (see also ~\cite[Eq.~(5)]{ChengZhuQi2009}):
\begin{equation}\label{eq:parity-decomposition}
    \|\mathcal{L}_A\|
    =\max\left\{
       \|\mathcal{L}_A|_{\mathbb{S}_n}\|,
       \|\mathcal{L}_A|_{\mathbb{K}_n}\|
    \right\}.
\end{equation}
Thus a maximizer in~\eqref{eq:norm} can always be chosen in one of these two
subspaces.  The question is whether it can always be chosen symmetric.

\begin{conjecture}[Symmetric-maximizer conjecture]\label{conjecture:byersnash}
For every $A\in\mathbb{R}^{n\times n}$,
\[
    \|\mathcal{L}_A|_{\mathbb{K}_n}\|
    \leq
    \|\mathcal{L}_A|_{\mathbb{S}_n}\|.
\]
Equivalently, the maximum in~\eqref{eq:norm} is attained at a symmetric
matrix.
\end{conjecture}

The statement originated as Theorem~9 of Byers and Nash~\cite{ByersNash1987},
up to replacing $A$ by $A^\top$.  Their main motivation
concerned the smallest singular value,
\begin{equation}\label{eq:min}
    \operatorname{sep}(A,-A^\top)
    =\min_{X\ne0}\frac{\|AX+XA^\top\|_F}{\|X\|_F},
\end{equation}
which is closely related to the conditioning of the Lyapunov equation.  They
proved, in particular, that a symmetric minimizer exists when $A$ is Hurwitz
stable, while the analogous statement fails for general $A$.

In 2015, Chen and Tian identified an error in the proof of the maximization
result and formulated it as a conjecture; they proved it for
$n\leq5$~\cite{ChenTian2015}.  The problem also appears
in~\cite{Cheng2001,ChengZhuQi2009}.  Feng, Lam, Yang, and Li proved the conjecture
for entrywise nonnegative, entrywise nonpositive, and tridiagonal
matrices~\cite{FengLamYangLi2015}.  As of 2021, the problem was described as open
for every $n\geq6$~\cite{KalantarovaTuncel2021}.  For $A,B\in
\mathbb{R}^{n\times n}$, the generalized continuous-time Lyapunov operator is
defined by
\[
    \mathcal{L}_{A,B}(X)=AXB^\top+BXA^\top,
    \qquad X\in\mathbb{R}^{n\times n}.
\]
Chen and Tian proved that the corresponding symmetric-maximizer statement for $\mathcal{L}_{A,B}$
holds for $n\leq3$ and gave a counterexample of order four~\cite{ChenTian2016}.

We disprove Conjecture~\ref{conjecture:byersnash} in every order $n\geq7$ by
giving an exact certificate in order seven and then applying a direct-sum
construction for any $n>7$.  This leaves order six as the only unresolved dimension.

\section{A counterexample of order seven}

\begin{theorem}\label{thm:counterexample}
Let
\begin{equation}\label{eq:A}
A=
\begin{pmatrix}
0&0&0&0&0&0&0\\
0&0&0&0&0&0&0\\
6&-5&-13&0&0&0&0\\
-14&-18&0&0&0&0&0\\
12&-12&11&0&0&0&0\\
0&0&0&-6&0&-14&0\\
0&0&0&6&-18&0&0
\end{pmatrix}.
\end{equation}
Then
\begin{equation}\label{eq:exact-separation}
    \left\|\mathcal{L}_A|_{\mathbb{S}_7}\right\|^2
    <1196<
    \left\|\mathcal{L}_A|_{\mathbb{K}_7}\right\|^2.
\end{equation}
In particular, $A$ is a counterexample to
Conjecture~\ref{conjecture:byersnash}.
\end{theorem}

\begin{proof}
Let $E_{ij}$ denote the elementary matrices.  Order a basis of $\mathbb{S}_7$ as
\[
E_{11},\ldots,E_{77},\quad
E_{12}+E_{21},E_{13}+E_{31},\ldots,E_{67}+E_{76},
\]
where the pairs $(i,j)$ with $i<j$ are ordered lexicographically.  Use the
similarly ordered basis
\[
E_{12}-E_{21},E_{13}-E_{31},\ldots,E_{67}-E_{76}
\]
of $\mathbb{K}_7$.  The
Gram matrices for these bases are
\[
    D_S=\operatorname{diag}(I_7,2I_{21}),
    \qquad
    D_K=2I_{21}.
\]
In addition, the respective coordinate matrices of $\mathcal{L}_A$ are denoted by
\[
    B_S\in\mathbb{Z}^{28\times28},
    \qquad
    B_K\in\mathbb{Z}^{21\times21}.
\]

Let $\mathbb{U}_7$ denote either $\mathbb{S}_7$ or $\mathbb{K}_7$. Likewise, $U$ is either $S$ or $K$.
If $X\in\mathbb{U}_7$ has coordinate vector $x$ in the corresponding basis,
then the definition of the Gram matrix gives
$\|X\|_F^2=x^\top D_U x$.  Moreover, $B_U x$ is the coordinate vector of
$\mathcal{L}_A(X)$ and so
$\|\mathcal{L}_A(X)\|_F^2=x^\top B_U^\top D_U B_U x$.  Taking the maximum over
all nonzero coordinate vectors gives
\begin{equation}\label{eq:rayleigh}
    \left\|\mathcal{L}_A|_{\mathbb{U}_7}\right\|^2
    =\max_{x\ne0}
      \frac{x^\top B_U^\top D_U B_U x}{x^\top D_U x}.
\end{equation}

For the symmetric restriction, set
\begin{equation}\label{eq:MS}
    M_S:=1196D_S-B_S^\top D_S B_S.
\end{equation}
Exact rational arithmetic yields a factorization
\[
    M_S=L\operatorname{diag}(d_1,\ldots,d_{28})L^\top,
    \qquad L\in\mathbb{Q}^{28\times28},
\]
where $L$ is unit lower triangular and $d_i\geq82$ for every $i$.  Hence
$M_S\succ0$.  Equations~\eqref{eq:MS} and~\eqref{eq:rayleigh} then give
\begin{equation}\label{eq:sym-upper}
    \left\|\mathcal{L}_A|_{\mathbb{S}_7}\right\|^2<1196.
\end{equation}

For the skew-symmetric restriction, consider the integer matrix
\[
K=\begin{pmatrix}
0&4&-9&41&71&4&6\\
-4&0&-11&-59&105&-12&3\\
9&11&0&59&-10&36&-8\\
-41&59&-59&0&9&0&2\\
-71&-105&10&-9&0&-15&0\\
-4&12&-36&0&15&0&4\\
-6&-3&8&-2&0&-4&0
\end{pmatrix}\in\mathbb{K}_7.
\]
Direct integer arithmetic gives
\[
    \|K\|_F^2=53836,
    \qquad
    \|\mathcal{L}_A(K)\|_F^2=64387950,
\]
and hence
\begin{equation}\label{eq:skew-lower}
    \left\|\mathcal L_A|_{\mathbb K_7}\right\|^2
    \geq
    \frac{\|\mathcal L_A(K)\|_F^2}{\|K\|_F^2}
    >1196.
\end{equation}
Combining~\eqref{eq:sym-upper} and~\eqref{eq:skew-lower}
proves~\eqref{eq:exact-separation}.  Finally,~\eqref{eq:parity-decomposition} shows
that the full norm is attained on the skew-symmetric subspace and cannot be
attained at a symmetric matrix.
\end{proof}

The accompanying Python program \texttt{verify\_counterexample.py} constructs
$B_S$ and $B_K$ from~\eqref{eq:A}, performs the $LDL^\top$ elimination over
$\mathbb{Q}$, verifies $d_i\geq82$ for every $i$, and checks all integer
identities above.

\begin{remark}\label{rem:numerics}
Floating-point computation gives
\[
    \left\|\mathcal{L}_A|_{\mathbb{S}_7}\right\|^2
       \approx1195.593985,
    \qquad
    \left\|\mathcal{L}_A|_{\mathbb{K}_7}\right\|^2
       \approx1196.025224.
\]
The integer $1196$ in~\eqref{eq:MS} was chosen because it lies
strictly between these two approximate values.  This numerical computation
serves only to motivate the exact certificate and is not used in the proof.
\end{remark}

\section{Counterexamples in every order
\texorpdfstring{$n\geq7$}{n greater than or equal to 7}}

\begin{corollary}\label{cor:higher-orders}
Conjecture~\ref{conjecture:byersnash} is false for every $n\geq7$.
\end{corollary}

\begin{proof}
For $m=n-7$, let $\widehat A=A\oplus0_m$, where $0_m$ denotes the
$m\times m$ zero matrix.  Writing a symmetric or skew-symmetric matrix
conformally with this block decomposition gives orthogonal decompositions into
the $7\times7$ part, the off-diagonal part, and the $m\times m$ part.  On the
off-diagonal part the Lyapunov operator acts as $Y\mapsto AY$, while it
vanishes on the $m\times m$ part.  Hence
\begin{align}
\left\|\mathcal L_{\widehat A}|_{\mathbb S_n}\right\|^2
&=\max\left\{
\left\|\mathcal L_A|_{\mathbb S_7}\right\|^2,\|A\|_2^2
\right\},\label{eq:padding-sym}\\
\left\|\mathcal L_{\widehat A}|_{\mathbb K_n}\right\|^2
&=\max\left\{
\left\|\mathcal L_A|_{\mathbb K_7}\right\|^2,\|A\|_2^2
\right\}.\label{eq:padding-skew}
\end{align}

After independent row and column permutations, the nonzero part of $A$ is the
orthogonal direct sum of
\[
A_1=\begin{pmatrix}
6&-5&-13\\
-14&-18&0\\
12&-12&11
\end{pmatrix},
\qquad
A_2=\begin{pmatrix}
-6&0&-14\\
6&-18&0
\end{pmatrix}.
\]
Therefore
\begin{align}
\|A\|_2^2
&=\max\{\|A_1\|_2^2,\|A_2\|_2^2\}\leq\max\{\|A_1\|_F^2,\|A_2\|_F^2\}\notag\\
&=\max\{1159,592\}<1196.\label{eq:A-bound}
\end{align}
Equations~\eqref{eq:exact-separation} and~\eqref{eq:padding-sym}--\eqref{eq:A-bound} imply
\[
    \left\|\mathcal L_{\widehat A}|_{\mathbb S_n}\right\|^2
    <1196<
    \left\|\mathcal L_{\widehat A}|_{\mathbb K_n}\right\|^2,
\]
which proves the claim.
\end{proof}

\section{Computational discovery and use of artificial intelligence}
\label{sec:computational-discovery}

The computational discovery used OpenAI's \texttt{gpt-5.6-sol} at high reasoning effort.  The model was given the
conjecture and was allowed to design the numerical test itself, rather than
being supplied with an objective function or an optimization method.  Its
initial approach led directly to the desired counterexample: exploit the
orthogonal splitting into symmetric and skew-symmetric matrices and numerically maximize
the gap
\[
    g(A):=
    \left\|\mathcal L_A|_{\mathbb K_n}\right\|
    -\left\|\mathcal L_A|_{\mathbb S_n}\right\|
\]
on the Frobenius unit sphere $\|A\|_F=1$.  For each trial matrix, the two
restricted norms were evaluated as largest singular values in orthonormal
coordinate bases.  The corresponding singular vectors supplied a gradient,
which was projected onto the tangent space of the sphere before the next
optimization step.

The first counterexample obtained in the completed search had order nine and
was found with Adam~\cite{KingmaBa2015} from nine Gaussian random starts.  The
moment parameters had the standard values $\beta_1=0.9$ and $\beta_2=0.999$.
The problem-specific stepsize was initially $0.02$ for $800$ iterations and
$0.006$ during refinement, with a quadratic decay to $15\%$ of its initial
value. All these values were chosen by the model.  The matrix was
normalized back to the Frobenius unit sphere after every step.  Manual
follow-up prompting then asked the same model to reduce the dimension and to
seek matrices with fewer nonzero entries and smaller coefficients.  Repeated
numerical searches and exact checks led to the sparse order-seven integer
matrix in~\eqref{eq:A}.

Given that the counterexample was found easily by numerical search, we investigated afterwards the role of the optimizer. 
It turns out that Adam's role was not incidental.  In later checks, the methods that worked most
consistently combined first-moment momentum with root-mean-square scaling of
the raw gradient.  Adam, and a few close variants, could escape the
large equality ridge $g(A)=0$; methods using only one of these ingredients did
not find a gap.  L-BFGS also failed from random
starts. This is
an empirical observation about this search on a few hundred random start matrices, not a general claim about the
optimizers.

The numerical computations served only to discover candidate matrices.  The
matrix in Theorem~\ref{thm:counterexample} and every inequality used in its
proof were subsequently verified in exact arithmetic, independently of the
floating-point search. 

\section{Conclusion}

Theorem~\ref{thm:counterexample} and Corollary~\ref{cor:higher-orders} settle
the symmetric-maximizer conjecture negatively in every order $n\geq7$.
Together with the positive result for $n\leq5$~\cite{ChenTian2015}, this leaves
only order six unresolved.  In particular, the present argument does not claim
that seven is the smallest order in which a counterexample can occur.

{\small
\bibliographystyle{alpha}
\bibliography{bib}
}

\end{document}